\documentclass[fleqn,12pt]{article}
\usepackage{eurofont}
\usepackage{multicol}
\usepackage{epic}
\usepackage{epsfig}
\usepackage[english]{babel}

\usepackage{amssymb,epsf,amsmath}
\newtheorem{Theorem}{Theorem}[section]

\newtheorem{Lemma}[Theorem]{Lemma}

\newtheorem{Ex}{Example}[section]
\newtheorem{Rem}{Remark}[section]

\newtheorem{proofhead}{Proof}

\newenvironment{proof}
    {
        \par
        \begin{proofhead}
        \normalfont
    }
    {   \qed
        \end{proofhead}
        \par
    }

\newcommand{\bb}{\mathbb}

\def \mean{{\bb E}}
\def \E{{\bb E}}

\def \var{{\bb V}{\rm ar}}

\def \cov{{\bb C}{\rm ov}}

\def \pr{{\bb P}}

\newcommand{\qed}{\hfill$\Box$}

\newcommand{\refs}[1]{(\ref{#1})}

\newcommand{\halmos}{\hfill $\Box$}

\newcounter{mylistcnt}
\renewcommand{\themylistcnt}{{\rm({\roman{mylistcnt}})}}

\newcounter{zad}

\begin{document}
\title{Exact asymptotics of supremum of a stationary Gaussian process over a random interval}
\author{Marek Arendarczyk\thanks{MA work was supported by MNiSW Grant N N201 412239 (2010-2011)},
Krzysztof D\c{e}bicki\thanks{KD work was supported by MNiSW
Research Grant N N201 394137 (2009-2011)}
\vspace*{.08in} \\
Mathematical Institute, University of Wroc\l aw \\
pl. Grunwaldzki 2/4, 50-384 Wroc\l aw, Poland }

\date{}
\maketitle

\begin{abstract}
Let $\{X(t) : t \in [0, \infty) \}$ be a centered stationary
Gaussian process. We study the exact asymptotics of $\pr (\sup_{s
\in [0,T]} X(t) > u)$, as $u \to \infty$, where $T$ is an
independent of \{X(t)\} nonnegative random variable. It appears
that the heaviness of $T$ impacts the form of the asymptotics,
leading to three scenarios: the case of integrable $T$, the case of $T$ having
regularly varying tail distribution with parameter
$\lambda\in(0,1)$ and  the case of $T$ having slowly varying tail distribution.


\end{abstract}


\section{Introduction}\label{s.intro}

Let
$\{X(t):t\ge0\}$ be a centered stationary Gaussian processes with covariance function
$r(t):=\cov(X(s+t),X(s))$.
One of the seminal results in the extreme value theory of
Gaussian processes is Pickands' exact asymptotics for the tail distribution of supremum
of $\{X(t):t\ge0\}$ over a given interval, say $[0,T]$,
\begin{equation}\label{pic}
        \pr(\sup_{s \in [0,T]} X(s) > u)
    =
        C^\frac{1}{\alpha} \mathcal{H}_\alpha u^\frac{1}{\alpha} \Psi(u) (1 + o(1)),
\end{equation}
as $u\to\infty$, providing
that $r(t) < 1$ for all $t > 0$ and
$r(t)=1-C|t|^\alpha + o(|t|^\alpha)$ as $t\to0$,
where $\alpha\in(0,2]$, $C>0$ and $\mathcal{H_\alpha}$ is the Picands' constant.

The aim of this paper is to give a
counterpart of \refs{pic}
for $T$ being an independent of $\{X(t):t\ge0\}$ nonnegative random variable.
The motivation of analysis of supremum distribution over a random-length
interval stems both from theoretic-level questions related with, e.g., extreme properties of
subordinated Gaussian processes (see, e.g., \cite{ArD10,Koz06})
and applied-level problems in such fields as queueing theory or ruin theory, see e.g. \cite{BDZ04}.

It appears that the form of the obtained asymptotics strongly depends on
{\it heaviness} of $T$, leading to three qualitatively different regimes:  the case of finite $\mean T$ ({\bf D1}),
the case of $T$ having regularly varying tail distribution with parameter $\lambda\in(0,1)$ ({\bf D2}),
and the case of $T$ having slowly varying tail distribution ({\bf D3}).

Each of the above cases needs its own approach, relying on the
interplay between variability of a Gaussian process
$\{X(t):t\ge0\}$ and heaviness of $T$. In particular, in case {\bf D1},
the asymptotics of \refs{pic} is mainly
inherited from the Gaussian structure of $\{X(t)\}$, while $T$ contributes to the
asymptotics only by its mean. The proof of this scenario is based
on the {\it double sum technique}; see Piterbarg \cite{Pit96} as a
key monograph in this field. A different approach is needed for
the study of cases {\bf D2}, {\bf D3}, where the heaviness of $T$
takes control over the form of the asymptotics.
The proof of this scenarios relies on an extension
of Theorem 12.3.4 in \cite{Lead83} (see Lemma 4.3).

The paper is organized as follows.
In Section \ref{s.notation} we introduce notation and precise the scenarios which are of our interest.
The main results of the paper are presented in Section \ref{s.results}.
All proofs are deferred to Section \ref{s.proofs}.

\section{Notation}\label{s.notation}
Let $\{X(t): t \in [0, \infty)\}$ be a centered stationary Gaussian process with a.s. continuous sample paths
and covariance function $r(t) := \cov(X(s), X(s + t))$.
We impose the following (subsets of) assumptions on $r(t)$: \\
\\
{\bf A1} $r(t) = 1 - C|t|^\alpha + o(|t|^\alpha)$, as $t \to 0$, with $\alpha \in (0,2]$ and $C>0$; \\
{\bf A2} $r(t) < 1$ for all $t > 0$;\\
{\bf A3} $r(t)\log(t) \to 0$ as $t \to \infty$.\\
\\
We are interested in the exact asymptotics of
\begin{eqnarray} \label{problem}
    \pr(\sup_{s \in [0, T]} X(s) > u)
\end{eqnarray}
as $u \to \infty$, where $T$ is an independent of $\{X(t)\}$ nonnegative random variable.

We distinguish three scenarios related to the heaviness of $T$:
\begin{itemize}
\item[{\bf D1}] $T$ is integrable, i.e., $\E T < \infty $ ;
\item[{\bf D2}] $T$ has regularly varying tail distribution with parameter $\lambda \in (0,1)$,
            i.e., $\pr(T > t) = L(t) t^{-\lambda} $, where $L(\cdot)$ is slowly varying at $\infty$;
\item[{\bf D3}] $T$ has slowly varying tail distribution, i.e., $\pr(T > t) = L(t)$,
            where $L(\cdot)$ is slowly varying at $\infty$.
\end{itemize}

In further analysis we use the following notation.
For given $\alpha\in(0,2]$,
by $\mathcal{H}_\alpha$ we denote the \emph{Pickands's constant}, defined by the following limit
\[
    \mathcal{H}_\alpha = \lim_{S \to \infty} \frac{\mathcal{H}_\alpha(S)}{S},
\]
where $\mathcal{H}_\alpha(S):=\E \exp\left(\sup_{t \in [0,S]}
\sqrt{2} B_\alpha(t) - t^{\alpha} \right)$, with
$\{B_\alpha(t):t\ge0\}$ being a standard fractional Brownian
motion with Hurst parameter $\alpha/2$. Moreover, let
$\Psi(u):=\pr(\mathcal{N}>u)$, with $\mathcal{N}$ denoting the
standard normal random variable. Recall that,
\begin{equation}        \label{Psi}
\Psi(u) = \frac{1}{\sqrt{2\pi}u}\exp\left(-u^2/2\right) (1+o(1)),\ \ {\rm as} \ \ u \to \infty.
\end{equation}

\section{Main results}\label{s.results}
In this section we present the asymptotics of \refs{problem}
for the described in Section \ref{s.intro}
cases {\bf D1, D2, D3}.
All proofs are deferred to Section \ref{s.proofs}.
We start with the analysis of regime {\bf D1}.
\begin{Theorem} \label{th.finite}
Let $\{X(t): t \in [0, \infty)\}$ be a centered stationary Gaussian process with covariance function $r(t)$
that satisfies {\bf A1-A2} and let $T$ be an independent of $\{X(t): t \in [0, \infty)\}$
nonnegative random variable that satisfies {\bf D1}. Then
\[
    \pr(\sup_{s \in [0,T]} X(s) > u) =   \E T C^\frac{1}{\alpha}\mathcal{H}_\alpha u^\frac{2}{\alpha} \Psi(u) (1 + o(1)),
\]
as $u \to \infty$.
\end{Theorem}

\begin{Rem}
The combination of \refs{pic} with Theorem \ref{th.finite}
leads to
\[
    \pr(\sup_{s \in [0,T]} X(s) > u)=  \E T  \pr(\sup_{s \in [0,1]} X(s) > u)(1+o(1)),
\]
as $u\to\infty$.
Hence, under {\bf D1}, the asymptotics of \refs{problem}
is mainly influenced by the Gaussian process $\{X(t)\}$, whereas
$T$ contributes only by its average behavior.
\end{Rem}

The next theorem deals with the case of $T$ that satisfies {\bf D2}.
For this scenario the
contribution of the event that $T$ is 'large' is crucial. Thus
we additionally need to assume ${\bf A3}$ (compare with, e.g., Theorem 12.3.4 in \cite{Lead83}).

\begin{Theorem}   \label{th.reg}
Let $\{X(t): t \in [0, \infty)\}$ be a centered stationary Gaussian process with covariance function
$r(t)$ that satisfies {\bf A1 - A3} and let $T$ be an independent of $\{X(t): t \in [0, \infty)\}$ nonnegative
random variable that satisfies {\bf D2}.
Then
\begin{eqnarray}
    \nonumber
    \lefteqn{
        \pr(\sup_{s \in [0,T]} X(s) > u)}\\
    \nonumber
    &=&
        \Gamma(1-\lambda)
        \frac{\mathcal{H}_\alpha^\lambda C^{\lambda/\alpha}}{(2\pi)^{\lambda/2}}
        L\left(
        u^\frac{\alpha - 2}{\alpha}
        \exp(u^2/2)\right)
        u^\frac{\lambda(2 - \alpha)}{\alpha}
        \exp\left( - \frac{\lambda u^2}{2}\right) (1 + o(1))
\end{eqnarray}
as
$       u \to \infty$.
\end{Theorem}
\begin{Rem}\label{rem2}
With $m(u):=\left[C^\frac{1}{\alpha} \mathcal{H}_\alpha u^\frac{2}{\alpha} \Psi(u) \right]^{-1}$
we can rewrite Theorem \ref{th.reg} in the form
\[
        \pr(\sup_{s \in [0,T]} X(s) > u)
    =
        \Gamma(1 - \lambda) \pr(T > m(u)) (1 + o(1))
\]
as $u \to \infty$. Thus, under {\bf D2}, the
heaviness of $T$ takes the control over the form
of the asymptotics.
\end{Rem}

Finally we turn on to scenario {\bf D3}.

\begin{Theorem}\label{th.veryheavy}
Let $\{X(t): t \in [0, \infty)\}$ be a centered stationary Gaussian process with covariance function
$r(t)$ that satisfies {\bf A1 - A3} and let $T$ be an independent of $\{X(t): t \in [0, \infty)\}$ nonnegative
random variable that satisfies {\bf D3}.
Then
\[
    \pr(\sup_{s \in [0,T]} X(s) > u)
    =
        L\left(
        u^\frac{\alpha - 2}{\alpha}
        \exp(u^2/2)\right)     (1 + o(1)),
\]
as $u \to \infty$.
\end{Theorem}

\begin{Rem}
Following notation introduced in Remark \ref{rem2}, under {\bf D3}, the asymptotics of
\refs{problem} takes the following form
\[
        \pr(\sup_{s \in [0,T]} X(s) > u)
    =
        \pr(T > m(u)) (1 + o(1)),
\]
as $u \to \infty$.
\end{Rem}

\section{Proofs}\label{s.proofs}

In this section we give detailed proofs of Theorems \ref{th.finite}, \ref{th.reg}, \ref{th.veryheavy}.
We begin with some auxiliary lemmas.

Let
$m(u):=\left[C^\frac{1}{\alpha} \mathcal{H}_\alpha u^\frac{2}{\alpha} \Psi(u) \right]^{-1}$.

The following lemma is a straightforward consequence of
Lemma 12.3.1 in \cite{Lead83}.

\begin{Lemma}  \label{12.3.1} 
Let $\varepsilon, a, A_\infty > 0$ and $q(u)= a u^{-2/\alpha}$.
Suppose that {\bf A1} and {\bf A3} hold.
Then
\[
         \frac{A_\infty m(u)}{q(u)}
    \sum_{\varepsilon \le kq(u) \le A_\infty m(u)} |r(kq(u))| \exp\left( -\frac{u^2}{1 + |r(kq(u))|}
    \right)
    \to
        0
\]
as $u \to \infty$.
\end{Lemma}
The following lemma follows from Lemma 12.2.11 in \cite{Lead83}.
\begin{Lemma}    \label{12.2.11} 
Suppose that {\bf A1} holds. Let $h > 0$ be fixed and such that
$\sup_{\varepsilon \le t \le h} r(t) < 1$ for each $\varepsilon > 0$.
Let $q(u) = a u^{-2/\alpha}$.
Then for each interval $I$ of length $h$,
\[
        0
    \le
        \pr(X(jq) \le u, jq(u) \in I) - \pr(\sup_{s \in I} X(s) \le u)
    \le
        h \rho(a) \frac{1}{m(u)} + o\left(\frac{1}{m(u)}\right)
\]
where, $\rho(a) \to 0$ as $a \to 0$.
\end{Lemma}
The next Lemma extends Theorem 12.3.4 in \cite{Lead83}, by showing that the considered
convergence is uniform on compact intervals. We note that using Slepian inequality,
we are able to give much shorter argument for the proof of the lower bound.

\begin{Lemma} \label{l.lim}
Let $\{X(t):t\ge0\}$ be a centered stationary Gaussian process with covariance function
that satisfies {\bf A1-A3}. Then for each $0 < A_0 < A_\infty < \infty$
\[
    \pr(\sup_{s \in [0,xm(u)]} X(s) \le u) \to e^{-x},
\]
as $u \to \infty$, uniformly for $x \in [A_0, A_\infty]$.
\end{Lemma}
%
%
\begin{proof}
Let $n_x := \left\lfloor xm(u) \right\rfloor$. \\
{\it Lower bound.} We prove that
\begin{equation} \label{lem.lower}
        \left( \pr(\sup_{s \in [0,1]} X(s) \le u) \right)^{n_x + 1}
    \le
        \pr(\sup_{s \in [0,xm(u)]} X(s) \le u),
\end{equation}
for each $u$ and $x\in[A_0,A_\infty]$.

Let
$ \{ X_i(t) , t \ge 0\}$, $i=0,1,...$, be independent copies of
$ \{ X(t) , t \ge0\}$  and
$ \{ Y(t) ; t \ge 0 \} $ be such that
$    Y(t) = X_i(t)$ for $t\in [i,i+1)$.

By Slepian inequality (see, e.g., Theorem C.1 in \cite{Pit96})
applied to processes $\{X(t)\}$ and $\{Y(t)\}$, we have
\begin{eqnarray}
    \nonumber
        \pr (\sup_{s \in [0, xm(u)]} X(s) \le u)
    &\ge&
        \pr (\sup_{s \in [0, n_x + 1]} X(s) \le u)\\
    \nonumber
    &\ge&
        \pr (\sup_{s \in [0, n_x + 1]} Y(s) \le u)  \\
    \nonumber
    &=&
        \left( \pr(\sup_{s \in [0, 1]} Y(s) \le u)  \right)^{n_x + 1}\\
    \nonumber
    &=&
        \left( \pr(\sup_{s \in [0, 1]} X(s) \le u)  \right)^{n_x + 1},
\end{eqnarray}
which proves \refs{lem.lower}.

{\it Upper bound.}
We show that
\begin{equation}       \label{lem.upper}
 \pr(\sup_{s \in [0,xm(u)]} X(s) \le u)
    \le
        \left( \pr(\sup_{s \in [0,1]} X(s) \le u) \right)^{n_x} (1 + o(1))
\end{equation}
as $u \to \infty$, uniformly for $x \in [A_0,A_\infty]$.

Let $\varepsilon > 0$.
We divide interval $[0, n_x]$ onto intervals of length $1$,
and split each of them onto subintervals $I_k^\star$, $I_k$ of length
$\varepsilon$, $1-\varepsilon$, respectively. \\
Let $a > 0$ and $q := q(u) = a u^{-\frac{2}{\alpha}}$. Observe that for $k = 1,2,...$
\[
        \pr\left(\sup_{s \in [0, n_x]} X(s) \le u\right)
    \le
        \pr\left(X(kq) \le u, kq \in \bigcup_{j = 1}^{n_x} I_j \right).
\]

In the first step we prove that
\begin{equation} \label{l.sum}
        \left|
        \pr(X(kq) \le u, kq \in \bigcup_{j = 1}^{ n_x } I_j )
        -
        \prod_{j = 1}^{n_x} \pr(X(kq) \le u, kq \in I_j)
        \right|
    \to
        0
\end{equation}
as $u \to \infty$ uniformly for $x \in [A_0, A_\infty]$.

To show (\ref{l.sum}) we proceed along similar lines to the proof of Theorem 8.2.4 in
\cite{Lead83}. Let $\Lambda = (\lambda_{ij})$ be the covariance matrix of
$X(kq), kq \in \cup_{j = 1}^{n_x} I_j$
and let $\Sigma = (\sigma_{ij})$ be the covariance matrix of
$Y(kq), kq \in \cup_{j = 1}^{n_x} I_j$ of independent standard normal random variables.
Applying Theorem 4.2.1 in \cite{Lead83}, we have
\begin{eqnarray}
\nonumber
\lefteqn{
        \left|
        \pr(X(kq) \le u, kq \in \bigcup_{j = 1}^{n_x} I_j )
        -
        \prod_{j = 1}^{n_x} \pr(X(kq) \le u, kq \in I_j)
        \right|      }\\
    \nonumber
    &=&
        \left|
        \pr(X(kq) \le u, kq \in \bigcup_{j = 1}^{n_x} I_j )
        -
        \pr(Y(kq) \le u, kq \in \bigcup_{j = 1}^{n_x} I_j )
        \right| \\
    &\le&   \label{matrix}
        \frac{1}{2\pi} \sum_{1\le i < j \le L} |\lambda_{ij} - \sigma_{ij}|
        (1 - \rho_{ij}^2)^{-\frac{1}{2}} \exp\left(- \frac{u^2}{1 + \rho_{ij}} \right),
\end{eqnarray}
where $L$ is the total number of $kq$ - points in $\cup_{j = 1}^{n_x} I_j$
and $\rho_{ij} = \max (|\lambda_{ij}|, |\sigma_{ij}|)$.
Now observe that, by the definition of the sequence $Y(kq)$ and matrix $\Sigma$, we have that
$|\lambda_{ii} - \sigma_{ii}| = 0$, $|\lambda_{ij} - \sigma_{ij}| = |r(kq)|$ for $k = i - j$.
Moreover, from the construction of the intervals $I_j$,
the minimum value of $kq$ is at least $\varepsilon$.
Combining the above with the observation that
$\sup\{|r(t)|;|t| \ge \varepsilon \} := \rho < 1$,
we arrive at an upper bound for (\ref{matrix})
\begin{eqnarray}
\nonumber
\lefteqn{
\frac{1}{2\pi(1 - \rho^2)^\frac{1}{2} } \frac{n_x}{q}
        \sum_{\varepsilon \le kq \le xm(u)} |r(kq)| \exp\left( -\frac{u^2}{1 + |r(kq)|} \right)}\\
    &\le&
        \frac{1}{2\pi(1 - \rho^2)^\frac{1}{2} } \frac{A_\infty m(u)}{q}
        \sum_{\varepsilon \le kq \le A_\infty m(u)} |r(kq)| \exp\left( -\frac{u^2}{1 + |r(kq)|} \right)\to 0,\label{nowe1}
\end{eqnarray}
as $u\to\infty$, where \refs{nowe1} is due to Lemma \ref{12.3.1}.
Hence (\ref{l.sum}) is satisfied.

In the second step we prove that
\begin{equation}    \label{lem.product}
        \limsup_{u \to \infty} \left| \prod_{j = 1}^{n_x} \pr(X(kq) \le u, kq \in I_j)
    -
        \left(\pr(\sup_{s \in [0,1]} X(s) \le u)\right)^{n_x}
        \right|
    \to
        0
\end{equation}
as $u \to \infty$, uniformly for $x \in [A_0, A_\infty]$.
In order to prove \refs{lem.product}, we use that
due to Lemma 27.1 in \cite{Bil95}, we have
\begin{eqnarray}
    \nonumber
        0
    &\le&
        \prod_{j = 1}^{n_x} \pr( X(kq) \le u, kq \in I_j)
    -
        \prod_{j = 1}^{n_x} \pr( \sup_{s \in I_j} X(s) \le u) \\
    \nonumber
    &\le&
        n_x \max_j \left( \pr( X(kq) \le u, kq \in I_j) -  \pr( \sup_{s \in I_j} X(s) \le u) \right) \\
    &\le&   \label{l.lim.max}
        (1 - \varepsilon)\frac{n_x}{m(u)} \rho(a) + n_x o\left( \frac{1}{m(u)} \right)  \\
    \nonumber
    &\le&
        (1 - \varepsilon)A_\infty \rho(a) + A_\infty m(u)o\left(\frac{1}{m(u)}\right)
    \to
        (1 - \varepsilon)A_\infty \rho(a)
    \le
        A_\infty \rho(a),
\end{eqnarray}
as $u \to \infty$.
Where (\ref{l.lim.max}) follows from Lemma \ref{12.2.11} 
with $\rho(a) \to 0$ as $a \to 0$.
Besides, the stationarity of $\{X(t)\}$ implies
\[
        \prod_{j = 1}^{n_x} \pr(\sup_{s \in I_k} X(s) \le u)
    =
        \left( \pr(\sup_{s \in I_1 } X(s) \le u) \right)^{n_x}
\]
and
\begin{eqnarray}
    \nonumber
        0
    &\le&
        \left( \pr(\sup_{s \in I_1 } X(s) \le u) \right)^{n_x}
        -
        \left( \pr(\sup_{s \in [0,1]} X(s) \le u) \right)^{n_x} \\
    \nonumber
    &\le&
    n_x \left(\pr(\sup_{s \in I_1 } X(s) \le u) - \pr(\sup_{s \in [0,1]} X(s) \le u) \right)\\
    \nonumber
    &\le&
        n_x \pr(\sup_{s \in I_1^\star} X(s) > u)    \\
    \nonumber
    &\le&
        A_\infty m(u) \pr(\sup_{s \in I_1^\star} X(s) > u) \\
    &=&  \label{I.star}
     \varepsilon A_\infty (1 + o(1))
\end{eqnarray}
as $u \to \infty$, where (\ref{I.star}) is
by Theorem D.2 in \cite{Pit96}. This confirms (\ref{lem.product}).

Now, in order to complete the proof, it suffices to combine (\ref{lem.lower}) and (\ref{lem.upper}) with the observation that
\[
        \lim_{u \to \infty} \left( \pr(\sup_{s \in [0,1]} X(s) <u)\right)^{n_x}
    =
        \lim_{u \to \infty} \left(1 - \frac{1}{m(u)} + o\left(\frac{1}{m(u)}\right)\right)^{x m(u)}
    =
        e^{-x},
\]
uniformly for $x \in [A_0, A_\infty]$.
\end{proof}


\subsection{Proof of Theorem \ref{th.finite}}

For given $R > 0$ and $u>0$, we introduce
$    \Delta_0 = \left[0, u^{-\frac{2}{\alpha}} R\right)$,
$    \Delta_k = \left[k u^{-\frac{2}{\alpha}} R, (k + 1)u^{-\frac{2}{\alpha}} R\right)$,
for $k = 1,2,...$,
$N_t = \left\lfloor  \frac{t}{u^{-\frac{2}{\alpha}} R} \right\rfloor$,
where $\left\lfloor \cdot \right\rfloor$ denotes integer part of a number.\\
{\it Upper bound.}
By stationarity of $\{X(t): t \in [0, \infty)\}$, we have
\begin{eqnarray}
        \pr(\sup_{s \in [0,T]} X(s) > u)
    \nonumber
    &\le&                        
        \int_0^\infty ( N_t + 1 )
        \pr(\sup_{s \in \Delta_0} X(s) > u) dF_T(t)\\
    &\le&
    \nonumber
        \pr(\sup_{s \in \Delta_0 } X(s) > u)
        \left( \frac{1}{R} u^\frac{2}{\alpha}
        \int_0^\infty t dF_T(t) + 1 \right)\\
    &=&                   \label{th1.le2}        
    \E T C^\frac{1}{\alpha}\frac{\mathcal{H}_\alpha(R)}{R}
     u^\frac{2}{\alpha} \Psi(u) (1 + o(1)),
\end{eqnarray}

as $u \to \infty$, where (\ref{th1.le2})
follows by Lemma D.1 in \cite{Pit96}. Thus, passing with $R \to \infty$,
we obtain the asymptotic upper bound
\[
        \pr(\sup_{s \in [0,T]} X(s) > u)
    \le
      \E T  C^\frac{1}{\alpha} \mathcal{H}_\alpha  u^\frac{2}{\alpha} \Psi(u) (1 + o(1))
\]
as $u \to \infty$.\\
%
{\it Lower bound.}
The idea of the proof of the lower bound is analogous to the proof of the lower bound
in Theorem D.2 in \cite{Pit96}. Hence we present only main steps of the argument.

Following Bonferroni inequality and stationarity of $\{X(t): t \in [0, \infty)\}$,
we have
\begin{eqnarray*}
        \pr(\sup_{s \in [0,T]} X(s) > u)
    &\ge&
    \nonumber
        \int_0^u \pr(\sup_{s \in [0,t]} X(s) > u) dF_T(t)\\
    &\ge&                           \label{th1.ge1}                 
        \int_0^u N_t \pr(\sup_{s \in \Delta_0} X(s) > u)    dF_T(t) \\
    &&-
    \nonumber
        \int_0^u \sum_{0\le i < j \le N_t}
        \pr(\sup_{s \in \Delta_i} X(s) > u , \sup_{s \in \Delta_j} X(s) > u) dF_T(t) \\
    &=&
    \nonumber
    I_1 - I_2.
\end{eqnarray*}
%
Observe that
\begin{eqnarray}
        I_1
    &\ge&
    \nonumber
      \pr(\sup_{s \in \Delta_0 } X(s) > u)
      \int_0^u \left( \frac{t}{u^{-\frac{2}{\alpha}}R} - 1 \right) dF_T(t)\\
    &\ge&
    \nonumber
        \pr(\sup_{s \in \Delta_0 } X(s) > u)
        \left( \frac{1}{R} u^\frac{2}{\alpha} \int_0^u t dF_T(t) - 1 \right)\\
    &=& \label{th.finite.I1}
        C^\frac{1}{\alpha}\frac{\mathcal{H}_\alpha(R)}{R} \E T u^\frac{2}{\alpha} \Psi(u) (1 + o(1))
\end{eqnarray}
as $u \to \infty$, where (\ref{th.finite.I1}) follows by Lemma D.1 in \cite{Pit96}.\\
Thus, having in mind that $R$ was arbitrary and
$\mathcal{H}_\alpha(R)/R\to\mathcal{H}_\alpha$,
it suffices to show that $I_2$ is asymptotically negligible.\\
Let $\varepsilon > 0$. Then
\begin{eqnarray}
        I_2
    &=&
    \nonumber
        \int_0^u \sum_{k = 1}^{N_t} (N_t - k)
        \pr(\sup_{s \in \Delta_0} X(s) > u , \sup_{s \in \Delta_k} X(s) > u) dF_T(t)\\
    &\le&
    \nonumber
        \frac{1}{R} u^\frac{2}{\alpha}
        \sum_{k = 1}^{N_u}
        \pr(\sup_{s \in \Delta_0} X(s) > u , \sup_{s \in \Delta_k} X(s) > u)
        \int_0^u t dF_T(t) \\
    &\le&
    \nonumber
        \frac{1}{R} \E T u^\frac{2}{\alpha}
        \pr(\sup_{s \in \Delta_0} X(s) > u , \sup_{s \in \Delta_1} X(s) > u) \\
    & + &
        \nonumber
        \frac{1}{R} \E T u^\frac{2}{\alpha}
        \sum_{k = 2}^{N_\frac{\varepsilon}{4}}
        \pr(\sup_{s \in \Delta_0} X(s) > u , \sup_{s \in \Delta_k} X(s) > u) \\
    &+&
        \nonumber
        \frac{1}{R} \E T u^\frac{2}{\alpha}
        \sum_{k = N_\frac{\varepsilon}{4}}^{N_u}
        \pr(\sup_{s \in \Delta_0} X(s) > u , \sup_{s \in \Delta_k} X(s) > u) \\
    \nonumber
    &=&
    I_3 + I_4 + I_5.
\end{eqnarray}

Following line by line the same argument as in the proof of Theorem D.2 in \cite{Pit96},
we conclude that $I_3$ and $I_4$ are negligible.
In order to bound $I_5$, we observe that
\[
        \pr(\sup_{s \in \Delta_0} X(s) > u, \sup_{t  \in \Delta_k} X(t) > u)
    \le
        \pr(\sup_{(t,s) \in \Delta_0 \times \Delta_k} X(s) + X(t) > 2u).
\]
Then, let $u$ be such that $R u^{-\frac{2}{\alpha}} \le \varepsilon/16$.
Hence the distance between
$\Delta_0$ and $\Delta_k$, for $k \ge N_\frac{\varepsilon}{4}$,
is not less then $\varepsilon/4$. Moreover, for each $s \in \Delta_0, t \in \Delta_k$
\[
        \var(X(s) + X(t))
    =
        2 + 2r(t-s)
    =
        4 - 2(1 - r(t - s))
    \le
        4 - \delta,
\]
with $\delta = 2 \inf_{s \ge \varepsilon/4} (1 - r(s)) > 0$
and
\[
        \E \sup_{(s,t) \in \Delta_0 \times\Delta_k} (X(s) + X(t))
    \le
        \E (\sup_{s \in \Delta_0} X(s) + \sup_{ t \in \Delta_k} X(t))
    =
        2   \E \sup_{s \in \Delta_0} X(s)
    \le
        a
\]
for some constant $a>0$ and each $k\in[N_{\varepsilon/4}, N_u]$.\\
Combining the above with Borell inequality (see, e.g., Theorem 2.1 in Adler \cite{Adl90}),
we obtain
\begin{eqnarray}
        \pr(\sup_{s \in \Delta_0} X(s) > u, \sup_{t \in \Delta_k} X(t) > u)
    \nonumber
    &\le&
        \pr(\sup_{(t,s) \in \Delta_0 \times \Delta_k} X(s) + X(t) > 2u) \\
    \nonumber
    &\le&
        \exp \left( - \frac{(u - a/2)^2}{2(1 - \delta/4)} \right).
\end{eqnarray}

Thus
\[
         I_5
    \le
         \frac{1}{R} \E T u^\frac{2}{\alpha} N_u \exp\left( - \frac{(u - a/2)^2}{2(1 - \delta/4)}\right),
\]
which in view of \refs{th.finite.I1} confirms that $I_5$ is asymptotically negligible.
This completes the proof.
\halmos

\subsection{Proof of Theorem \ref{th.reg}}
Let $0 < A_0 < A_\infty$. We make the following decomposition
\begin{eqnarray}
        \pr(\sup_{s \in [0,T]} X(s) > u)
    \nonumber
    &=&
        \int_0^{A_0 m(u)} \pr(\sup_{s \in [0,t]} X(s) > u) dF_T(t)\\
    \nonumber
    &+&
        \int_{A_0 m(u)}^{A_\infty m(u)} \pr(\sup_{s \in [0,t]} X(s) > u) dF_T(t)\\
    \nonumber
    &+&
        \int_{A_\infty m(u)}^\infty \pr(\sup_{s \in [0,t]} X(s) > u) dF_T(t)\\
    \nonumber
    &=&
    I_1 + I_2 + I_3.
\end{eqnarray}
We analyze each of the integrals $I_1, I_2, I_3$ separately. \\

{\it Integral $I_1$}.
Due to the stationarity of the process $\{X(t):t\ge0\}$, we have
\begin{eqnarray}
    \nonumber
       I_1&\le&
        \pr(\sup_{s \in [0,1]} X(s) > u)
        \left[ \int_0^{A_0 m(u)} t  dF_T(t) + 1 \right]\\
    \nonumber
    &=&
        \pr(\sup_{s \in [0,1]} X(s) > u)
        \left[ \int_0^{A_0 m(u)} \pr(T > t)  dt - A_0 m(u) \pr(T > A_0 m(u)) + 1 \right].\\
        \label{I1.1}
\end{eqnarray}
Applying Karamata's theorem (see, e.g., Proposition 1.5.8 in \cite{Bin87}) we have, as $u \to \infty$
\[
        \int_0^{A_0 m(u)} \pr(T > t) dt
    =
        \frac{1}{1-\lambda} A_0 m(u) \pr(T > A_0 m(u))  (1 + o(1)),
\]
which combined with (\ref{I1.1}) and Theorem D.2 in \cite{Pit96},
implies the following asymptotical upper bound of $I_1$
\[
    I_1
    \le
    \frac{\lambda}{1-\lambda} A_0 \pr(T > A_0 m(u)) (1 + o(1))
  =
    \frac{\lambda}{1-\lambda} A_0^{1-\lambda} \pr(T > m(u)) (1 + o(1)),
\]
as $u \to \infty$.\\
{\it Integral $I_3$}. We have
\[
        I_3 \le \pr(T > A_\infty m(u))
    =
        A_\infty^{-\lambda} \pr( T > m(u) )(1 + o(1))
\]
as $u \to \infty$.

{\it Integral $I_2$}.
Let $\varepsilon > 0$. Due to Lemma \ref{l.lim}, for sufficiently large $u$,
we get the following upper bound
\begin{eqnarray}
    \lefteqn{    \nonumber
        I_2
    =
        \int_{A_0}^{A_\infty} \pr(\sup_{s \in [0,x m(u)]} X(s) > u) dF_T(x m(u))}\\
    \nonumber
    &\le&
        (1 + \varepsilon)\int_{A_0}^{A_\infty} (1 - e^{-x}) dF_T(x m(u))    \\
    \nonumber
    &=&
        (1 + \varepsilon) [ \int_{A_0}^{A_\infty} e^{-x} \pr(T > xm(u)) dx
    \nonumber
    -
        (1 - e^{-A_\infty})\pr(T > A_\infty m(u))\\
    \nonumber
    &&+
        (1 - e^{-A_0})\pr(T > A_0 m(u))  ] .
\end{eqnarray}
In an analogous way, for the lower bound, we obtain, for sufficiently large $u$,
\begin{eqnarray}
        I_2
    \nonumber
    &\ge&
        (   1 - \varepsilon) [ \int_{A_0}^{A_\infty} e^{-x} \pr(T > xm(u)) dx  \\
    \nonumber
    &&-
        (1 - e^{-A_\infty})\pr(T > A_\infty m(u)) \\
    \nonumber
    &&+
        (1 - e^{-A_0})\pr(T > A_0 m(u))  ] .
\end{eqnarray}

Due to {\bf D2} combined with Theorem 1.5.2 in \cite{Bin87}, we have
\[
        \int_{A_0}^{A_\infty} e^{-x} \pr(T > xm(u)) dx
    =
        \pr(T > m(u)) \int_{A_0}^{A_\infty} e^{-x} x^{-\lambda}  dx (1 + o(1))
\]
as $u \to \infty$.

Thus
for each $\varepsilon > 0,  A_\infty > A_0 > 0$,
\begin{eqnarray*}
    \liminf_{u \to \infty} \frac{I_2}{\pr(T > m(u))}\ge
        (1 - \varepsilon) [ \int_{A_0}^{A_\infty} e^{-x} x^{-\lambda}  dx
        -
        (1 - e^{-A_\infty}) A_\infty^{-\lambda} + (1 - e^{-A_0}) A_0^{-\lambda} ]
\end{eqnarray*}
and
\begin{eqnarray*}
\limsup_{u \to \infty} \frac{I_2}{\pr(T > m(u))}
    \le
        (1 + \varepsilon) [  \int_{A_0}^{A_\infty} e^{-x} x^{-\lambda}  dx
        -
        (1 - e^{-A_\infty}) A_\infty^{-\lambda} + (1 - e^{-A_0}) A_0^{-\lambda} ].
\end{eqnarray*}

Hence, passing with
$A_0 \to 0$, $A_\infty \to \infty$ and $\varepsilon\to0$, we conclude that
$I_1$ and $I_3$ are negligible and
\[
I_2=\Gamma(1-\lambda)\pr(T > m(u))(1+o(1)),
\]
as $u\to\infty$,
which in view of {\bf D2}, definition of $m(u)$ and \refs{Psi}, completes the proof.
\halmos

\subsection{Proof of Theorem \ref{th.veryheavy}}

{\it Lower bound.}
From Lemma \ref{l.lim}, for given $A_\infty>0$,
\begin{eqnarray}
        \pr(\sup_{s \in [0,T]} X(s) > u)
    \nonumber
    &\ge&
        \pr(\sup_{s \in [0,A_\infty m(u)]} X(s) > u)
        \pr( T > A_\infty m(u) ) \\
    \nonumber
    &=&
        (1 - e^{-A_\infty})
        \pr(T > m(u)) (1 + o(1)),
\end{eqnarray}
as $u \to \infty$.
Thus, passing with $A_\infty\to\infty$, we get that
\[
\pr(\sup_{s \in [0,T]} X(s) > u)\ge\pr(T > m(u)) (1 + o(1)),
\]
as $u\to\infty$.\\
{\it Upper bound.}
For the upper bound we have
\[
    \pr(\sup_{s \in [0,T]} X(s) > u)
    \le
    \int_0^{A_0 m(u)} \pr(\sup_{s \in [0,t]} X(s) > u) dF_T(t)
    +
        \pr(T > m(u)).
\]
Due to the stationarity of the process $\{X(t):t\ge0\}$ we have
\begin{eqnarray}
                I_1
        \nonumber
        & =: &
        \int_0^{A_0 m(u)} \pr(\sup_{s \in [0,t]} X(s) > u) dF_T(t)\\
    \nonumber
    &\le&
        \pr(\sup_{s \in [0,1]} X(s) > u)
        \left[\int_0^{A_0 m(u)} t dF_T(t)  + 1\right] \\
    \nonumber
    &=&
        \pr(\sup_{s \in [0,1]} X(s) > u)
        \left[\int_0^{A_0 m(u)} \pr(T > t) dt - A_0m(u) \pr(T > A_0 m(u))  + 1\right] \\
    &\le&       \label{slowly.I1}
        \pr(\sup_{s \in [0,1]} X(s) > u)
        \left[\int_0^{A_0 m(u)} \pr(T > t) dt + 1\right]
\end{eqnarray}
Applying Karamata's theorem (see, e.g., Proposition 1.5.8 in \cite{Bin87}) we have, as $u \to \infty$
\[
        \int_0^{A_0 m(u)} \pr(T > t) dt
    =
        A_0 m(u) \pr(T > A_0 m(u))  (1 + o(1)),
\]
which combined with (\ref{slowly.I1}) and Theorem D.2 in \cite{Pit96},
implies the following asymptotical upper bound of $I_1$
\[
        I_1
    \le
        A_0 \pr(T > A_0 m(u)) (1 + o(1))
    =
        A_0 \pr(T > m(u)) (1 + o(1))
\]
as $u \to \infty$.
Thus,
\[
        \pr(\sup_{s \in [0,T]} X(s) > u)
    \le
        (1 + A_0) \pr(T > m(u)) (1 + o(1)),
\]
as $u\to\infty$.
In order to complete the proof it suffices to pass with $A_0\to 0$.
\halmos

\end{document}